\documentclass[12pt]{amsart}
\usepackage{amsmath, amssymb, graphicx, amsthm}
\usepackage[colorlinks=true, allcolors=blue]{hyperref}
\usepackage[capitalise]{cleveref}
\usepackage{tikz-cd}
\usepackage[hmarginratio={1:1}]{geometry}
\usepackage{color}
\usepackage[backend=bibtex8,style=alphabetic]{biblatex}
\usepackage{mathrsfs, eufrak}
\usepackage{parskip}
\usepackage{fancyhdr}
\usepackage{enumitem}
\usepackage{caption}

\newtheorem{thm}{Theorem}[section]
\newtheorem{lem}[thm]{Lemma}

\newtheorem{defn}[thm]{Definition}

\numberwithin{equation}{section}
\allowdisplaybreaks

\newcommand{\beq}{\begin{equation}}
\newcommand{\eeq}{\end{equation}}
\newcommand{\bal}{\begin{align*}}
\newcommand{\eal}{\end{align*}}

\newcommand{\p}{\partial}
\newcommand{\D}{\nabla}

\newcommand{\la}{\langle}
\newcommand{\ra}{\rangle}

\newcommand{\R}{\mathbb R}

\newcommand{\eps}{\varepsilon}

\newcommand{\Span}{\operatorname{Span}}

\newcommand{\Seq}{\Subset}

\title{Sobolev Inequalities in Spacelike Submanifolds of Minkowski Space}
\date{}

\author{Liang Xu}
\address{School of Mathematical Sciences, Zhejiang University, Hangzhou 310058, China}
\email{liang.xu@zju.edu.cn}

\begin{document}

\maketitle

\begin{abstract}
We follow the method of ABP estimate in \cite{brendle2021} and apply it to spacelike submanifolds in $\R^{n,1}$. We then obtain Michael-Simon type inequalities. Surprisingly, our investigation leads to a Sobolev inequality without a mean curvature term, provided the hypersurface is mean convex. 
\end{abstract}


\section{Introduction}

In this paper we are mainly concerned with a specific type of Sobolev inequality for spacelike submanifolds in Minkowski space. Associated to a spacelike submanifold $\Sigma^m\hookrightarrow\R^{n,1}$, we define the maximal slope by
\[
\tau(\Sigma)=\sup\{|\nu_0(x)|:x\in\Sigma, \nu(x) \text{ is a unit normal to } \Sigma \text{ at } x\}. 
\]
See \cref{def:2.1} for details. The main results are as follows. 

\begin{thm}\label{thm:11}
Suppose $\Sigma^n\subseteq\R^{n,1}$ is a smooth, compact and spacelike hypersurface. Assume that $\Sigma$ is mean convex and that $f$ is any smooth and positive function defined on $\Sigma$. Then
\beq\label{eq:11}
\int_\Sigma|\nabla f| + \int_{\p\Sigma}f \geq C_{n,\tau}\left(\int_\Sigma f^{\frac{n}{n-1}} \right)^{\frac{n-1}{n}}, 
\eeq
where the constant $C_{n,\tau} = n\omega_n^{\frac{1}{n}}(n+1)^{-\frac{1}{n}}\tau^{-\frac{1}{n}}(\tau+\sqrt{\tau^2-1})^{-1}$. 
\end{thm}

Under the condition of mean convexity, there is no mean curvature term involved, which, to our knowledge, is new, Without the assumption of mean convexity, similar result holds with a curvature term involved. 
\begin{thm}\label{thm:1.2new}
Suppose $\Sigma^n\subseteq\R^{n,1}$ is a smooth, compact and spacelike hypersurface. Assume that $f$ is any smooth and positive function defined on $\Sigma$. Then
\beq\label{eq:1.2}
\int_\Sigma\sqrt{|\nabla f|^2+f^2\|H\|^2} + \int_{\p\Sigma}f \geq C_{n,\tau}\left(\int_\Sigma f^{\frac{n}{n-1}} \right)^{\frac{n-1}{n}}, 
\eeq
with $C_{n,\tau}=2^{-1}n\omega_n^{\frac{1}{n}}(n+1)^{-\frac{1}{n}}\tau^{-\frac{1}{n}}(\tau+\sqrt{\tau^2-1})^{-1}$.
\end{thm}

Here $\|H\|^2=-|H|^2\geq0$; see \cref{sec:2} for details. Next we establish the same Sobolev inequality for submanifold $\Sigma^m\subseteq \mathbb R^{n,1}$ of higher codimension $n-m+1$, with $0<m<n$. Let $\nu$ be a normal vector field of $\Sigma$ with $|\nu|^2=-1$. We then write $T_x^\bot\Sigma=T_x^{\bot,1}\Sigma\oplus T_x^{\bot,2}\Sigma$, where $T_x^{\bot,2}\Sigma=\Span\{\nu(x)\}$. Accordingly, the mean curvature vector is decomposed into $H^{\bot,1}+H^{\bot,2}$. 

\begin{thm}\label{thm:13}
Suppose $\Sigma^m\subseteq\R^{n,1}$ is a smooth, compact and spacelike submanifold. Let $\nu$ be any normal vector field of $\Sigma$ with $|\nu|^2=-1$ and $f$ any smooth and positive function defined on $\Sigma$. Then
\beq\label{eq:12}
\int_\Sigma \sqrt{|\nabla f|^2+f^2(|H^{\bot,1}|^2+\|H^{\bot,2}\|^2)} + \int_{\p\Sigma}f \geq C_{n,m,\tau}\left(\int_\Sigma f^{\frac{m}{m-1}} \right)^{\frac{m-1}{m}},
\eeq
where the constant $C_{n,m,\tau}=m2^{-\frac{n}{m}}(n+1)^{-\frac{1}{m}}\omega_n^\frac{1}{m}\tau^{-\frac{1}{m}}(\tau+\sqrt{\tau^2-1})^{-\frac{n}{m}}$. 
\end{thm}

At the end of this work, we discovered that in \cite{Tsai2022} similar results were established. Nevertheless, compared to the results therein, our construction yields a Sobolev inequality without curvature terms for a mean convex hypersurface. The history of geometric inequalities probably dates back to ancient Greece. The classical isoperimetric inequality asserts that for a domain $\Sigma\subseteq\R^n$ with sufficiently well-behaved boundary, there holds 
\beq\label{eq:TheIsoperimetricInequality}
|\p\Sigma|\geq n\omega_n^{\frac{1}{n}}|\Sigma|^{\frac{n-1}{n}}. 
\eeq
It is known that such an isoperimetric inequality is essentially equivalent to the $W^{1,1}$ Sobolev inequality for domains in Euclidean space: 
\beq\label{eq:1.4new}
\int_\Sigma|\D f|+\int_{\p\Sigma} f \geq n\omega_n^{\frac{1}{n}}\left( \int_\Sigma f^\frac{n}{n-1}  \right)^\frac{n-1}{n}. 
\eeq

Intensive work has been established to extend \cref{eq:TheIsoperimetricInequality} or \cref{eq:1.4new} to more general settings. It has long been conjectured that the same inequality \cref{eq:TheIsoperimetricInequality} holds in Cartan-Hadamard manifolds \cite{Aubin1976}. 
Partial results include \cite{weil1926surfaces,Beckenbach1933, Croke1984, Kleiner1992}. See also \cite{ghomiTotalCurvatureIsoperimetric2021} for a recent attempt to resolve the conjecture. It is also possible to replace areas and volumes in \cref{eq:TheIsoperimetricInequality} by more general quermassintegrals. The resulting isoperimetric inequality is proved for hypersurfaces with certain convexity in Euclidean space. See \cite{Guan2012} for details.

For a domain $\Sigma$ in a two dimensional space form of constant curvature $K$, there is a neat result which states that
\beq\label{eq:TwoDimensionalSpaceForms}
4\pi |\Sigma|\leq |\p\Sigma|^2+K|\Sigma|^2. 
\eeq
Please see \cite{Choe2005} and references therein. The same inequality \cref{eq:TwoDimensionalSpaceForms} is proved by Choe and Gulliver \cite{Choe1992} for minimal surfaces $\Sigma^2$ with certain topological constraints in hyperbolic space $\mathbb H^n$. Yau \cite{Yau1975}, Choe and Gulliver \cite{Let1992} showed that if $\Sigma$ is a domain in $\mathbb H^n$ or a $n$-dimensional minimal submanifold in $\mathbb H^{n+m}$, then it satisfies the linear isoperimetric inequality
\[
(n-1)|\Sigma|\leq|\p\Sigma|. 
\]
Another linear inequality for proper minimal submanifolds in $\mathbb H^n$ is obtained in \cite{Min2014} using Poincar\'e model. 

It is a longstanding conjecture that \cref{eq:TheIsoperimetricInequality} holds true for minimal hypersurfaces in $\R^{n+1}$. Using the method of sliding, Brendle fully settled the problem in a recent work \cite{brendle2021}, and later extended his results to Riemannian manifolds with nonnegative Ricci curvature \cite{Brendle2020}. The most classical application of the sliding method is perhaps Aleksandrov's maximum principle. In  \cite{Cabre2008} Cabr\'e first employed the sliding method and gave a simple and elegant proof of \cref{eq:TheIsoperimetricInequality}. 

We follow Brendle's method and apply it to submanifolds in the Minkowski space, and obtain some Michael-Simon-Sobolev type inequalities. 

\textbf{Acknowledgements.} The author would like to thank Research Professor Qi-Rui Li for his instructions and many helpful discussions.

\section{Notations and Preliminaries}\label{sec:2}

Let $\R^{n,1}$ be the Minkowski space endowed with metric 
\[
\bar g=-dx_0^2+dx_1^2+\cdots+dx_n^2. 
\]
The usual Euclidean metric of $\R^{n+1}$ is denoted by $\delta$. The volume element of $\R^{n,1}$, $d\mu_{\bar g}=\sqrt{-\det\bar g} dX$, is just the usual Lebesgue measure. A vector $Y$ is called unit if $|Y|^2=\sigma(Y)$, where $\sigma(Y)$ is the signature of $Y$, i.e. $\sigma(X)=1,-1,0$ if $X$ is spacelike, timelike, lightlike, respectively. Finally, we define $\|Y\| = \sqrt{\sigma(Y)|Y|^2}$. 

Let $(\Sigma^n,g)\hookrightarrow(\R^{n,1},\bar g)$ be a spacelike hypersuface. We denote by $x$ a point in $\Sigma$, and by $X(x)$ the corresponding position vector in $\R^{n,1}$. Anything with a `bar' is a quantity of  the ambient space. Then the second fundamental form is defined by 
\[
\bar\nabla_{Y}Z=\nabla_{Y}Z-h(Y,Z), \quad \forall Y,Z\in \mathscr X(\Sigma). 
\]
Let $\nu(x)$ be a normal to $\Sigma$ at the point $x$. Clearly $\nu(x)$ is also a vector in $\R^{n,1}$. Denote by $\nu_\alpha(x)$ the $\alpha$-th coordinate of $\nu(x)$ in $\R^{n,1}$, so that
\[
|\nu|^2=-\nu_0^2+\nu_1^2+\cdots+\nu_n^2. 
\]
\begin{defn}\label{def:2.1}
Associated to a spacelike submanifold $\Sigma^m\hookrightarrow\R^{n,1}$, the maximal slope is defined by
$
\tau(\Sigma) = \sup \{ |\nu_0(x)|: x\in\Sigma, \nu(x)\in T_x^\bot\Sigma, |\nu(x)|^2=-1 \}. 
$
\end{defn}
The quantity $\tau$ characterizes how `lightlike' $\Sigma$ is; for instance, $\tau=1$ if $\Sigma^n\Seq\R^n$; and $\tau$ is uniquely determined by the diameter if $\Sigma^n\Seq\mathbb H^n$.

\begin{lem}\label{lem:21}
For any function $w$ on the ambient space, $\nabla^2 w=\bar\nabla^2 w - \la h,\bar\nabla w \ra_{\bar g}$. 
\end{lem}

\begin{proof}
We assume that $T_x\R^{n,1}$ is spanned by orthonormal basis $\{\eps_\alpha: 0\leq \alpha\leq n\}$, and $T_x\Sigma$ by $\{\eps_i: 1\leq i\leq n\}$, with $|\eps_i|^2=1, |\eps_0|^2=-1$. We then compute 
\begin{align*}
\bar\nabla^2_{ij}w &= \p^2_{ij}w-\la \bar\D_{\eps_i}\eps_j,\bar\D w \ra\\
&= \p^2_{ij}w-\la \bar\Gamma_{ij}^k\eps_k+\bar\Gamma_{ij}^2\eps_0,w^\ell\eps_\ell+w^0\eps_0 \ra \\
&= \p^2_{ij}w-\bar\Gamma_{ij}^k\p_kw+\bar\Gamma_{ij}^0\p_0w. 
\end{align*}
By Gauss formula $\bar\nabla_{\eps_i}\eps_j=\nabla_{\eps_i}\eps_j-h_{ij}\eps_0$ we see that $\bar\Gamma_{ij}^k=\Gamma_{ij}^k$ and $\bar\Gamma_{ij}^0=-h_{ij}$. Hence
\[
\bar\nabla^2_{ij}w=\nabla^2_{ij}w-h_{ij}\p_0w=\nabla^2_{ij}w+\la h(\eps_i,\eps_j),\bar\nabla w \ra
\]
By tensorality, the same formula holds in any coordinate systems.
\end{proof}

 \section{Proof of \cref{thm:11}}\label{sec:3}

Since \cref{eq:11} is homogeneous in $f$, by normalization we may assume that
\beq\label{eq:31}
\int_{\p\Sigma} f = n\int_\Sigma f^{\frac{n}{n-1}}-\int_\Sigma|\nabla f|. 
\eeq
Let $\eta$ be the outward unit normal of $\p\Sigma$. Consider the following PDE
\beq\label{eq:32}
\begin{cases}
\operatorname{div}(f\nabla u) = nf^{\frac{n}{n-1}} - |\nabla f|, & \text{ in } \Sigma\setminus\p\Sigma,\\
\la \nabla u,\eta \ra_{\bar g} = 1, & \text{ along } \p\Sigma. 
\end{cases}
\eeq
By our normalization, the equation has a solution $u\in C^{2,\alpha}$. Since $\Sigma$ is mean convex, we may assume that the mean curvature vector $H$ is either zero or timelike pointing to the past. Let $\nu$ be the unit normal and timelike vector field pointing to the past. Now fix $r>0$. For any $x\in\Sigma$ we define $\mathcal A_{r}=\cap_{x\in\Sigma}\Lambda_{r}(x)$, where
\begin{align}\label{eq:3.3new}
\Lambda_{r}(x)=\left\{ p\in\R^{n,1}: |(p-X(x))^\top|^2<r^2, -r\leq\la p-X,\|H\|^{-1}H(x) \ra\leq0 \right\}. 
\end{align}
Note that when $\|H\|=0$, the notation $\|H\|^{-1}H(x)$ simply means the normal $\nu(x)$. 

\begin{figure}[t]
\centering
\includegraphics[scale=0.6]{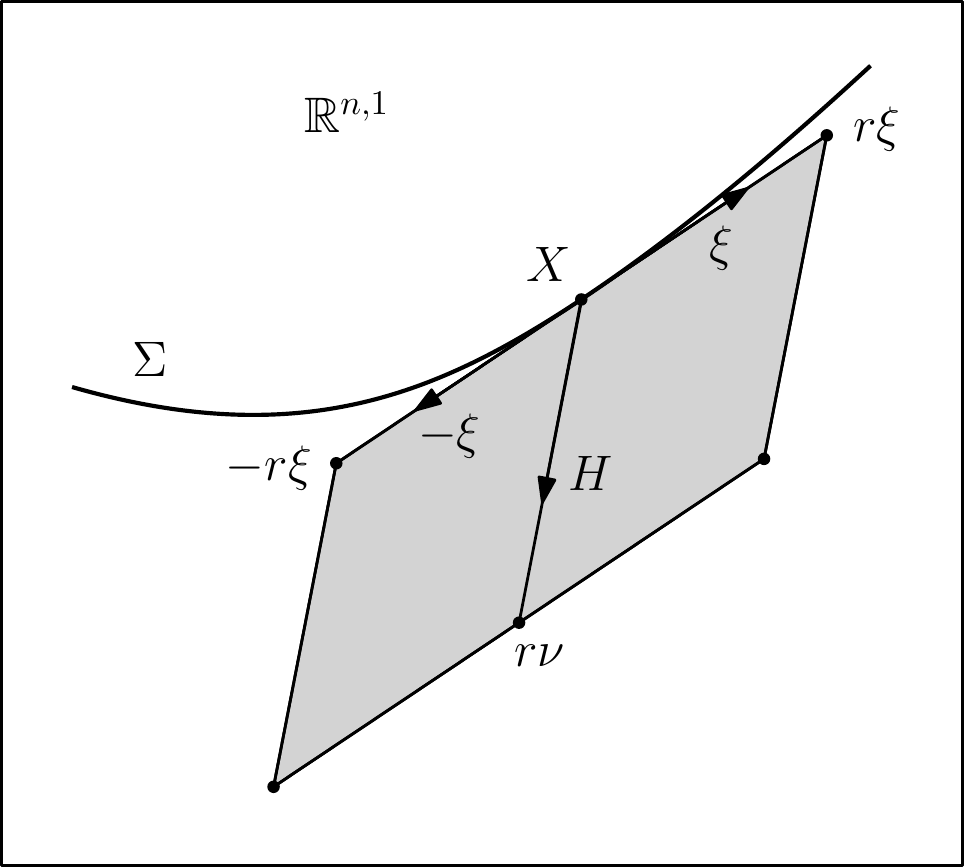}
\caption{An illustration of $\Lambda_{r}(x)$.}
\label{fig:1}
\end{figure}

We write $p-X=s\xi+t\nu$, where $\xi$ is a unit tangent vector and $\nu$ is a unit normal vector pointing to the past. The the definition of $\Lambda_r(x)$ implies that $-r<s<r,0\leq t\leq r$. Therefore $\Lambda_r(x)$ is a parallelogram illustrated as in \cref{fig:1}. We continue and define 
\begin{align*}
&\mathcal U=\left\{ (x,y): x\in\Sigma\setminus\p\Sigma, y\in T^{\bot}_x\Sigma, |\nabla u(x)|<1, -1\leq\la y,\|H\|^{-1}H(x) \ra\leq0 \right\}, \\
&\mathcal B_{r} = \left\{ (x,y)\in\mathcal U: r\nabla^2u(x)+r\la h(x),y \ra+g(x)\geq0 \right\}. 
\end{align*}
Finally, we take the map $\Phi_r:T^\bot\Sigma\to\R^{n,1}$, 
\beq
\Phi_r(x,y)=X(x)+r(\nabla u(x)+y). 
\eeq

\begin{lem}\label{lem:31}
We have asymptotic behavior
\beq
\liminf_{r\to\infty}r^{-n-1}{|\mathcal A_r|} \geq \tilde C_{n,\tau}, 
\eeq
where $|\mathcal A_r|$ is the usual Lebesgue measure in $\R^{n+1}$ and $\tilde C_{n,\tau} = \frac{\omega_n}{(n+1)\tau(\tau+\sqrt{\tau^2-1})^n}$. 
\end{lem}

\begin{proof}
We blow down $\mathcal A_r$ by factor $r$. As $r\to\infty$, the bounded domain $\Sigma$ collapses to a single point: the origin, and each $\Lambda_r(x)$ converges to a $\tilde\Lambda(x)$, specified by
\[
\tilde\Lambda(x) = \{ \tilde p\in \R^{n,1}: |\tilde p^\top|^2<1, -1<\la \tilde p,\nu(x) \ra<0 \}, 
\]
where $\nu(x)\in-\mathbb H^n$ is the unit normal to $\Sigma$ at $x$ pointing to the past. See the left of \cref{fig:2} for an illustration. Clearly $\tilde\Lambda(x)$ contains the cone $\tilde{\mathcal C}(x)=\{\tilde p\in\R^{n,1}:|\tilde p|^2<0, -1<\la \tilde p,\nu(x) \ra<0\}$. 
\begin{figure}[t]
\centering
\includegraphics[scale=0.6]{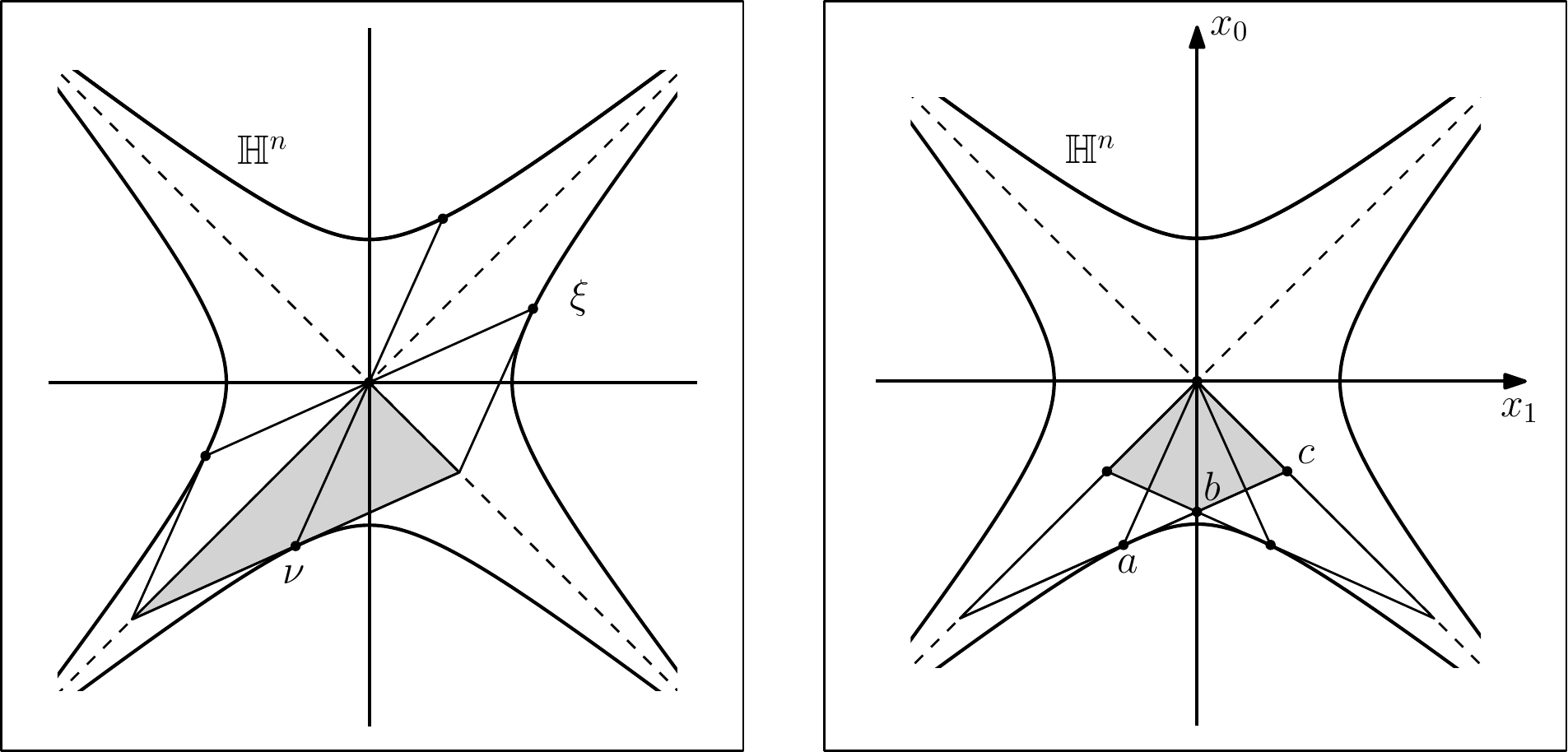}
\caption{Left: blowing down $\Lambda(x)$ to get $\tilde\Lambda(x)$, which contains $\tilde{\mathcal C}(x)$, the shaded area. Right: $\tilde{\mathcal A}$ contains at least $\tilde{\mathcal C}$, the shaded area.}
\label{fig:2}
\end{figure}
Therefore $\tilde{\mathcal A}=\cap_{x\in\Sigma}\tilde\Lambda(x)$ contains at least $\tilde {\mathcal C} = \cap_{x\in\Sigma}\tilde{\mathcal C}(x)$. Since by assumption $\nu(x)$ has minimal height $-\tau$, we may assume that $\tilde{\mathcal C}$ is the union of two cones, as illustrated by the right of \cref{fig:2}, with the zeroth coordinate of point $a$ being $-\tau$. 

We now proceed by computing the volume of $\tilde{\mathcal C}$. Without loss of generality, we assume that the points $a,b,c$ lie in the plane $Ox_0x_1$. Then $a=(-\tau,-\sqrt{\tau^2-1},0,\cdots,0)$, and the tengential $\vec{ac}$ is parallel to $\xi=(\sqrt{\tau^2-1},\tau,0,\cdots,0)$. From this we readily obtain $b=(-\tau^{-1},0,\cdots,0)$ and $c=(-\tau+\sqrt{\tau^2-1},\tau-\sqrt{\tau^2-1},0,\cdots,0)$. Consequently 
\[
|\tilde{\mathcal C}| = \omega_n(\tau-\sqrt{\tau^2-1})^n\cdot \tau^{-1}\cdot(n+1)^{-1} = \tilde C_{n,\tau}
\]
Hence $\liminf_{r\to\infty}r^{-n-1}|\mathcal A_r|\geq|\tilde{\mathcal C}|=\tilde C_{n,\tau}$. 
%
\end{proof}

\begin{lem}
There holds $\Phi_r(\mathcal B_{r})\supseteq\mathcal A_{r}$. 
\end{lem}

\begin{proof}
For any given $p\in\mathcal A_r$, consider the function
\[
F(x)=ru(x)+\frac{1}{2}|p-X(x)|^2, \quad x\in\Sigma. 
\]
By compactness $F$ attains its minimum at some $\bar x\in\Sigma$. We claim that $\bar x\notin\p\Sigma$. For if otherwise $\bar x\in\p\Sigma$, then at this point
$
0\geq\la \nabla F(\bar x),\eta \ra = r - \la p-X(\bar x),\eta \ra. 
$
We have by definition of $\mathcal A_r$ that $\la p-X(\bar x),\eta \ra = \la p-X(\bar x)^\top,\eta \ra<r$, a contradiction. Hence $\bar x\in\Sigma\setminus\p\Sigma$ and at which $\nabla F(\bar x)=r\nabla u(\bar x)-(p-X(\bar x))^\top=0$. We then find $\bar y\in T_{\bar x}^\bot\Sigma$ such that $r\bar y=(p-X(\bar x))^\bot$. Obviously $r|\nabla u(\bar x)|=|(p-X(\bar x))^\top|<r$ and $-r\|H\|\leq r\la \bar y,H(\bar x) \ra=\la p-X(\bar x),H(\bar x) \ra\leq0$. Finally, by \cref{lem:21},
\begin{align*}
0 &\leq \D^2F(\bar x) = r\D^2u(\bar x) + \bar\D^2\left(\frac{1}{2}|p-X(\bar x)|^2\right) - \left\la h,\bar\D\left(\frac{1}{2}|p-X(\bar x)|^2\right) \right\ra \\
&= r\D^2u(\bar x) + g(\bar x) - \la h,X(\bar x)-p \ra \\
&= r\D^2u(\bar x) + g(\bar x) + r\la \bar y,h \ra, 
\end{align*}
completing the proof. 
\end{proof}

In the Riemannian or Lorentzian setting, in order for the area formula to be true, the Jacobian of a map should be modified with volume elements; that is, 
\[
J(\Phi_r) = |\det(D\Phi_r)|\cdot\frac{\sqrt{-\det\bar g}}{\sqrt{\det g}}, 
\]
where $D\Phi_r$ is the usual tangent map of the coordinate map. 

\begin{lem}
The invariant Jacobian of $\Phi_r$ is given by
\[
J(\Phi_r)=r^{n+1}\det\left(\nabla_i\D^ju(x)+\la y,h_i^j \ra+\delta_i^j/r\right). 
\]
\end{lem}

\begin{proof}
At a fixed point $(x,y)$, we pick an orthonormal basis $\{e_i,\nu\}$ that spans $T_{(x,y)}(T^\bot\Sigma)$, and a normal coordinate system $\{x_i,y\}$ such that $\frac{\p}{\p x_i}=e_i,\frac{\p}{\p y}=\nu$ at $(x,y)$. Then $\Phi_r(x,y)=X+r\D u+ry\nu$. We now compute
\begin{align*}
&\la \frac{\p\Phi_r}{\p x_i},e_j \ra =  \la e_i+r\bar\D_{e_i}\D u+{ry\bar\D_{e_i}\nu},e_j \ra = \delta_{ij}+ru_{ij}-ryh_{ij}, \\
&\la \frac{\p\Phi_r}{\p y},e_j \ra = \la r\nu,e_j \ra = 0, \\
&\la \frac{\p\Phi}{\p y},\nu \ra = \la r\nu,\nu \ra = -r. 
\end{align*}
Note that we have used the fact that $\la \bar\D_{e_i}e_j,\nu \ra=\la -h_{ij}\nu,\nu \ra=h_{ij}$ and that $\la \bar\D_{e_i}\nu,e_j \ra=-\la \bar\D_{e_i}e_j,\nu \ra=-h_{ij}$. Thus $\det(D\Phi_r)=-r\det(ru_{ij}-ryh_{ij}+\delta_{ij})$, whence
\[
J(\Phi_r) = r^{n+1}\det(\nabla_i\D^ju+\la y,h_i^j \ra+\delta_i^j/r). 
\]
By tensorality the same formula holds in any coordinate system. 
\end{proof}

\begin{lem}\label{lem:34}
We have for any $(x,y)\in\mathcal B_{r}$
\[
\det\left(\nabla_i\D^ju(x)+\la y,h_i^j \ra+\delta_i^j/r\right) \leq \left(f^{\frac{1}{n-1}}+1/r\right)^n. 
\]
\end{lem}

\begin{proof}
By definition, $\nabla^2u+\la y,h \ra+g/r\geq0$ for any $(x,y)\in\mathcal B_{r}$. Therefore
\[
\det\left(\nabla_i\D^ju+\la y,h_i^j \ra+\delta_i^j/r\right) \leq \left(\frac{\Delta u+\la y,H \ra}{n}+\frac{1}{r}\right)^n \leq \left(\frac{\Delta u}{n}+\frac{1}{r}\right)^n. 
\]
On the other hand, from \cref{eq:32} and the fact that $|\nabla u|<1$, we have
\[
\Delta u = nf^{\frac{1}{n-1}}-f^{-1}\left( |\nabla f|+\la \nabla f,\nabla u \ra \right) \leq nf^{\frac{1}{n-1}}, 
\]
completing the proof. 
\end{proof}

\begin{proof}[Proof of \cref{thm:11}]
By area/coarea formula and previous lemmas, we have
\[
\frac{|\mathcal A_{r}|}{r^{n+1}} \leq \int_{\mathcal B_{r}} \left(f^{\frac{1}{n-1}}+1/r\right)^n dyd\mu_g \leq \int_{\Sigma}\left(f^{\frac{1}{n-1}}+1/r\right)^n d\mu_g. 
\]
Sending $r\to\infty$, we derive
\beq
\tilde C_{n,\tau} \leq \int_{\Sigma}f^{\frac{n}{n-1}} d\mu_g. 
\eeq
Combining this and \cref{eq:31}, we conclude that
\[
\int_{\p\Sigma} f + \int_\Sigma|\nabla f|= n\left(\int_\Sigma f^{\frac{n}{n-1}}\right)^{\frac{1}{n}}\left(\int_\Sigma f^{\frac{n}{n-1}}\right)^{\frac{n-1}{n}} \geq n\tilde C_{n,\tau}^{\frac{1}{n}}\left(\int_\Sigma f^{\frac{n}{n-1}}\right)^{\frac{n-1}{n}}. 
\]
We finally write
$
C_{n,\tau} = n\tilde C_{n,\tau}^{\frac{1}{n}} = n\omega_n^{\frac{1}{n}}(n+1)^{-\frac{1}{n}}\tau^{-\frac{1}{n}}(\tau+\sqrt{\tau^2-1})^{-1}. 
$
\end{proof}

\section{Proof of \cref{thm:1.2new}}

Without the `mean convexity' assumption, the union $\mathcal A_r=\cap_{x\in\Sigma}\Lambda_r(x)$, with $\Lambda_r(x)$ defined by \cref{eq:3.3new}, might as well be empty. We therefore construct $u$ by
\beq\label{eq:4.1new}
\begin{cases}
\operatorname{div}(f\D u) = nf^\frac{n}{n-1} - \sqrt{|\D f|^2+f^2\|H\|^2}, & \text{ in } \Sigma, \\
\la \D u,\eta \ra_{\bar g} = 1, & \text{ along } \p\Sigma. 
\end{cases}
\eeq
Since \cref{eq:1.2} is homogeneous in $f$, we may assume 
\beq\label{eq:4.2new}
\int_{\p\Sigma} f = \int_{\Sigma}nf^\frac{n}{n-1} - \int_\Sigma\sqrt{|\D f|^2+f^2\|H\|^2}, 
\eeq
so that $\cref{eq:4.1new}$ admits a solution. We then modify \cref{eq:3.3new} by
\[
\Lambda_{r}(x)=\left\{ p\in\R^{n,1}: |(p-X(x))^\top|^2<\frac{r^2}{4}, -\frac{r}{2}\leq\la p-X(x),\|H\|^{-1}H(x) \ra\leq\frac{r}{2} \right\}
\]
and $\mathcal A_r=\cap_{x\in\Sigma}\Lambda_r(x)$. Accordingly, 
\begin{align*}
&\mathcal U=\left\{ (x,y): x\in\Sigma\setminus\p\Sigma, y\in T^{\bot}_x\Sigma, |\nabla u(x)|<\frac{1}{2}, -\frac{1}{2}\leq\la y,\|H\|^{-1}H(x) \ra\leq\frac{1}{2} \right\}, \\
&\mathcal B_{r} = \left\{ (x,y)\in\mathcal U: r\nabla^2u(x)+r\la h(x),y \ra+g(x)\geq0 \right\}, 
\end{align*}
and $\Phi_r(x,y)=X(x)+r(\D u(x)+y)$. Note that when $\|H\|=0$, the notation $\|H\|^{-1}H$ simply means the normal vector $\nu$. Similar as in \cref{sec:3}, we still have
\begin{itemize}
\item The inclusion $\Phi_r(\mathcal B_r)\supseteq\mathcal A_r$; and
\item The Jacobian $J(\Phi_r)=r^{n+1}\det\left(\nabla_i\D^ju(x)+\la y,h_i^j \ra+\delta_i^j/r\right)$. 
\end{itemize}
The proofs are almost identical. 

\begin{lem}\label{lem:4.1}
We have asymptotic behavior
\beq
\liminf_{r\to\infty}r^{-n-1}{|\mathcal A_r|} \geq \tilde C_{n,\tau}, 
\eeq
where $\tilde C_{n,\tau} = 2^{-n}(n+1)^{-1}\omega_n\tau^{-1}(\tau+\sqrt{\tau^2-1})^{-n}$. 
\end{lem}

\begin{proof}
\begin{figure}[t]
\centering
\includegraphics[scale=0.6]{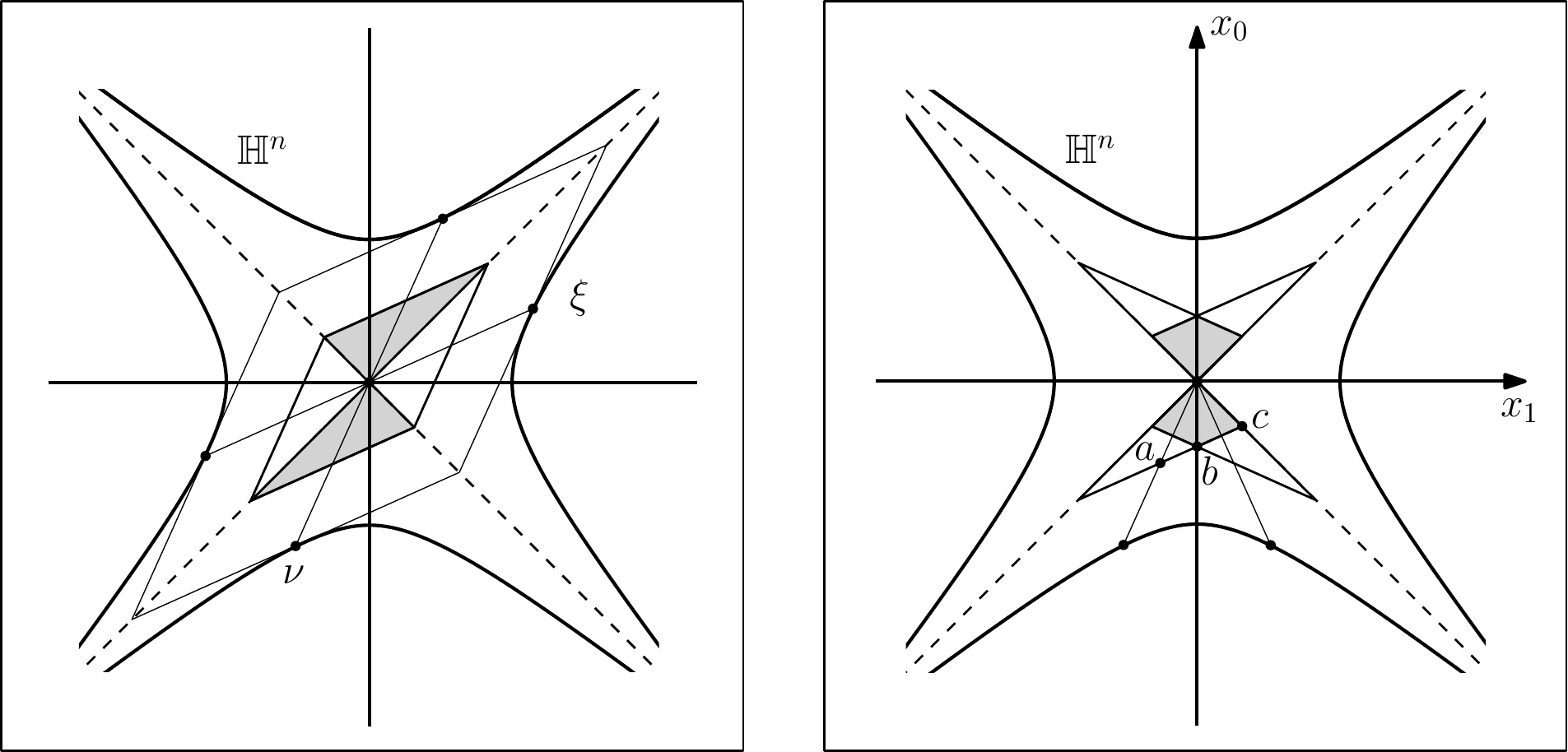}
\caption{Left: each $\tilde\Lambda(x)$ contains two cones, the shaded area; Right: $\tilde\Lambda$ contains at least $\tilde {\mathcal C}$, the shaded area.}
\label{fig:3new}
\end{figure}
The proof is similar to that of \cref{lem:31}, with minor differences. Since the minimal height of unit normals is $-\tau$, in \cref{fig:3new}, $a=(-\frac{\tau}{2}, -\frac{\sqrt{\tau^2-1}}{2},0,\cdots,0)$. The vector $\vec{ac}$ is parallel to $\xi=(\sqrt{\tau^2-1},\tau,0,\cdots,0)$. From this we derive $b=(-\frac{1}{2\tau}, 0\cdots,0)$ and $c=(\frac{-\tau+\sqrt{\tau^2-1}}{2},\frac{\tau-\sqrt{\tau^2-1}}{2},0\cdots,0)$. Thus $\tilde{\mathcal C}$ has volume
\[
|\tilde{\mathcal C}| = \omega_n\left( \frac{\tau-\sqrt{\tau^2-1}}{2} \right)^n\cdot\frac{1}{2\tau}\cdot\frac{1}{n+1}\cdot2 = 2^{-n}(n+1)^{-1}\omega_n\tau^{-1}(\tau+\sqrt{\tau^2-1})^{-n}, 
\]
from which it follows that
\[
\liminf_{r\to\infty}r^{-n-1}{|\mathcal A_r|} \geq |\tilde{\mathcal C}| \geq \tilde C_{n,\tau}, 
\]
completing the proof. 
\end{proof}

\begin{lem}
We have for any $(x,y)\in\mathcal B_{r}$
\[
\det\left(\nabla_i\D^ju(x)+\la y,h_i^j \ra+\delta_i^j/r\right) \leq \left(f^{\frac{1}{n-1}}+1/r\right)^n. 
\]
\end{lem}

\begin{proof}
By definition, $\nabla^2u+\la y,h \ra+g/r\geq0$ for any $(x,y)\in\mathcal B_{r}$. Therefore
\[
\det\left(\nabla_i\D^ju+\la y,h_i^j \ra+\delta_i^j/r\right) \leq \left(\frac{\Delta u+\la y,H \ra}{n}+\frac{1}{r}\right)^n. 
\]
Moreover, we have
\begin{align*}
\la y,fH \ra - \la \nabla f,\nabla u \ra &\leq \frac{1}{2}f\|H\| + |\D f||\D u| \leq \sqrt{|\D f|^2+f^2\|H\|^2}\sqrt{|\D u|^2+1/4} \\
&\leq \sqrt{|\D f|^2+f^2\|H\|^2}\sqrt{1/4+1/4} \leq \sqrt{|\D f|^2+f^2\|H\|^2}. 
\end{align*}
Combined with \cref{eq:4.1new}, it follows that
\[
\Delta u + \la y,H \ra = nf^{\frac{1}{n-1}}-f^{-1}\left( \sqrt{|\nabla f|^2+f^2\|H\|^2}+\la \nabla f,\nabla u \ra - \la y,fH \ra \right) \leq nf^{\frac{1}{n-1}}, 
\]
giving the assertion. 
\end{proof}

\begin{proof}[Proof of \cref{thm:1.2new}]
By area/coarea formula, we have
\[
\frac{|\mathcal A_r|}{r^{n+1}} \leq \int_{\mathcal B_r}\frac{J(\Phi_r)}{r^{n+1}}dyd\mu_g \leq \int_\Sigma\left( f^\frac{1}{n-1}+1/r \right)^nd\mu_g. 
\]
Sending $r\to\infty$, we derive $\int_\Sigma f^\frac{n}{n-1}\geq\tilde C_{n,\tau}$. In the view of \cref{eq:4.2new}, we compute
\begin{align*}
&\int_{\p\Sigma}f + \int_\Sigma\sqrt{|\D f|^2+f^2\|H\|^2} = n\int_\Sigma f^\frac{n}{n-1} \\
&= n\left(\int_\Sigma f^\frac{n}{n-1}\right)^\frac{n-1}{n}\left(\int_\Sigma f^\frac{n}{n-1}\right)^\frac{1}{n} \geq n\tilde C_{n,\tau}^{\frac{1}{n}}\left(\int_\Sigma f^\frac{n}{n-1}\right)^\frac{n-1}{n}. 
\end{align*}
Finally we write 
$
C_{n,\tau} = n\tilde C_{n,\tau}^{\frac{1}{n}} = n2^{-1}(n+1)^{-\frac{1}{n}}\omega_n^{\frac{1}{n}}\tau^{-\frac{1}{n}}(\tau+\sqrt{\tau^2-1})^{-1}. 
$
\end{proof}

\section{Proof of \cref{thm:13}}

We fix such a section $\nu\in\Gamma(T^\bot\Sigma)$, with $|\nu|^2=-1$; that is, $\nu$ is a unit timelike normal vector field along $\Sigma$. At each point $x\in\Sigma$, we have decompostion
\[
T_x^{\bot}\Sigma = T_x^{\bot,1}\Sigma \oplus T_x^{\bot,2}\Sigma, 
\]
where $T_x^{\bot,2}\Sigma = \Span\{\nu(x)\}$. Note that $T_x^{\bot,1}\Sigma$ is a spacelike subspace. Accordingly, we write a normal vector as $y=y^{\bot,1}+y^{\bot,2}$. The equation we consider now is 
\beq\label{eq:5.1}
\begin{cases}
\operatorname{div}(f\D u) = mf^\frac{m}{m-1} - \sqrt{|\D f|^2+f^2(|H^{\bot,1}|^2+\|H^{\bot,2}\|^2)}, & \Sigma\setminus\p\Sigma, \\
\la \D u,\eta \ra = 1, & \p\Sigma. 
\end{cases}
\eeq
We normalize by
\beq\label{eq:5.2}
\int_{\p\Sigma} f = \int_{\Sigma}mf^\frac{m}{m-1} - \int_\Sigma\sqrt{|\D f|^2+f^2(|H^{\bot,1}|^2+\|H^{\bot,2}\|^2)}
\eeq
so that the PDE has a solution. The domains and codomains are
\begin{align*}
&\Lambda_r(x) = \left\{p\in\R^{n,1}: |(p-X(x))^\top|^2+|(p-X(x))^{\bot,1}|^2<\frac{r^2}{4}, -\frac{r}{2}\leq\la p-X(x),\nu(x) \ra\leq \frac{r}{2}\right\}, \\
&\mathcal A_r = \cap_{x\in\Sigma}\Lambda_r(x), \\
&\mathcal U = \left\{ (x,y): x\in\Sigma\setminus\p\Sigma, y\in T_x^\bot\Sigma, |\D u(x)|^2+|y^{\bot,1}|^2<\frac{1}{4}, -\frac{1}{2}\leq\la y^{\bot,2},\nu \ra\leq\frac{1}{2} \right\}, \\
&\mathcal B_r = \{ (x,y)\in\mathcal U: r\D^2u(x)+r\la y,h(x)+g(x)\geq0 \ra \}. 
\end{align*}
Finally we take $\Phi_r: T^\bot\Sigma\to\R^{n,1}: (x,y)\mapsto X(x)+r(\D u(x)+y)$. As in \cref{sec:3}, 
\begin{itemize}
\item The inclusion $\Phi_r(\mathcal B_r)\supseteq\mathcal A_r$; and
\item The Jacobian $J(\Phi_r)=r^{n+1}\det\left(\nabla_i\D^ju(x)+\la y,h_i^j \ra+\delta_i^j/r\right)$. 
\end{itemize}
The proofs are almost identical. 

\begin{lem}
We have asymptotic behavior
\beq
\liminf_{r\to\infty}r^{-n-1}{|\mathcal A_r|} \geq \tilde C_{n,\tau}, 
\eeq
where $\tilde C_{n,\tau} = 2^{-n}(n+1)^{-1}\omega_n\tau^{-1}(\tau+\sqrt{\tau^2-1})^{-n}$. 
\end{lem}

\begin{proof}
We write $\R^{n,1}=T_x\Sigma\oplus T_x^{\bot}\Sigma=(T_x\Sigma\oplus T_x^{\bot,1}\Sigma)\oplus T_x^{\bot,2}\Sigma$. If $p\in\Lambda_r(x)$, then
\[
|(p-X(x))^\top|^2+|(p-X(x))^{\bot,1}|^2<\frac{r^2}{4} \quad \text{ and } \quad |(p-X(x))^{\bot,2}|^2 \leq \frac{r^2}{4}. 
\]
If $\xi\in T_x\Sigma\oplus T_x^{\bot,1}\Sigma$ is a unit vector and $p-X=s\xi+t\nu$, then $|s|,|t|\leq\frac{r}{2}$, the shape of $\Lambda_r(x)$ and $\mathcal A_r$ would be exactly like those in \cref{fig:3new}. Thus the same conclusion holds as in \cref{lem:4.1}. 
\end{proof}

\begin{lem}
We have for any $(x,y)\in\mathcal B_{r}$
\[
\det\left(\nabla_i\D^ju(x)+\la y,h_i^j \ra+\delta_i^j/r\right) \leq \left(f^{\frac{1}{m-1}}+1/r\right)^m. 
\]
\end{lem}

\begin{proof}
By geometric-arithmetic inequality, 
\[
\det\left(\nabla_i\D^ju(x)+\la y,h_i^j \ra+\delta_i^j/r\right) \leq \left( \frac{\Delta u+\la y,H \ra}{m} + \frac{1}{r} \right)^m. 
\]
Now by \cref{eq:5.1}, we have
\begin{align*}
\Delta u+\la y,H \ra &= mf^\frac{1}{m-1} - f^{-1}(\sqrt{|\D f|^2+f^2(|H^{\bot,1}|^2+\|H^{\bot,2}\|^2)}+\la \D f,\D u \ra - \la fH,y \ra). 
\end{align*}
On the other hand, by the definition of $\mathcal B_r$, 
\begin{align*}
-\la \D f,\D u \ra + \la fH,y \ra &= \la \D f,\D u \ra + \la fH^{\bot,1},y^{\bot,1} \ra + \la fH^{\bot,2},y^{\bot,2} \ra \\
&\leq |\D f||\D u| + |fH^{\bot,1}||y^{\bot,1}| + \frac{1}{2}f\|H^{\bot,2}\| \\
&\leq \sqrt{|\D f|^2+f^2|H^{\bot,1}|^2+f^2\|H^{\bot,2}\|^2}\sqrt{|\D u|^2+|y^{\bot,1}|^2+\frac{1}{4}} \\
&\leq \sqrt{|\D f|^2+f^2(|H^{\bot,1}|^2+\|H^{\bot,2}\|^2)}. 
\end{align*}
This implies that $\Delta u+\la y,H \ra\leq mf^\frac{1}{m-1}$, completing the proof. 
\end{proof}

\begin{proof}[Proof of \cref{thm:13}]
Using area/coarea formula, we have
\[
\frac{|\mathcal A_r|}{r^{n+1}} \leq \int_\Sigma\left( f^\frac{1}{m-1}+1/r \right)^m. 
\]
Taking $r\to\infty$, we obtain
\[
\tilde C_{n,\tau} \leq \int_\Sigma f^\frac{m}{m-1}. 
\]
Taking \cref{eq:5.2} into consideration, we derive
\begin{align*}
m\tilde C_{n,\tau}^{\frac{1}{m}} \left( \int_{\Sigma}f^\frac{m}{m-1} \right)^\frac{m-1}{m} &\leq m\left( \int_{\Sigma}f^\frac{m}{m-1} \right)^\frac{1}{m}\left( \int_{\Sigma}f^\frac{m}{m-1} \right)^\frac{m-1}{m} = m\int_{\Sigma}f^\frac{m}{m-1} \\
&= \int_{\p\Sigma} f + \int_\Sigma\sqrt{|\D f|^2+f^2(|H^{\bot,1}|^2+\|H^{\bot,2}\|^2)}. 
\end{align*}
Finally, we compute $C_{m,n,\tau} = m\tilde C_{n,\tau}^{\frac{1}{m}} = m2^{-\frac{n}{m}}(n+1)^{-\frac{1}{m}}\omega_n^\frac{1}{m}\tau^{-\frac{1}{m}}(\tau+\sqrt{\tau^2-1})^{-\frac{n}{m}}$. 
\end{proof}

\newpage

\printbibliography

\end{document}